\documentclass{article}
\usepackage[utf8]{inputenc}

\usepackage{amsmath}
\usepackage{amssymb}
\usepackage{amsthm}

\AtEndDocument{\noindent DEPARTMENT OF MATHEMATICAL SCIENCES, SEOUL NATIONAL UNIVERSITY, SEOUL, 151-747, KOREA. E-mail adress: sagelee24@naver.com
}

\newtheorem{theorem}{Theorem}[section]
\newtheorem{lemma}[theorem]{Lemma}
\newtheorem{corollary}[theorem]{Corollary}

\newtheorem{remark}[theorem]{Remark}

\newtheorem{notation}[theorem]{Notation and Assumption}

\numberwithin{equation}{section}

\newcommand{\vertiii}[1]{{\left\vert\kern-0.25ex\left\vert\kern-0.25ex\left\vert #1 
    \right\vert\kern-0.25ex\right\vert\kern-0.25ex\right\vert}}

\title{The Hyperinvariant Subspace Problem}
\author{SA GE LEE}

\date{June 2023}

\begin{document}

\maketitle

\begin{abstract}
The hyperinvariant subspace problem is solved in the setting of Hilbert and right Hamilton space, motivated by my earlier works in the invariant subspace problem ([7]--[14]).
\end{abstract}

\section{Introduction}\label{sec:1}
Throughout the paper, $T$ will be a fixed, but arbitrarily chosen bounded linear operator on a separable infinite dimensional Hilbert space over $\mathbb{C}$ or $\mathbb{R}$, or right Hmailton space $H$ ([14]). 
The hyperinvariant subspace problem: 
Does $T$ have a nontrivial closed subspace of $H$, invariant under for every $A \in \{T \}'$? ( cf. (1.2) below).

\begin{notation}\label{notation:1.2}{\rm

\begin{align}\label{eq:1.1}
B(H) = \textup{The algebra of bounded linear operator on $H$, equipped}\\
\textup{with the operator norm $\| \cdot \|$, the strong operator topology (SO)}\nonumber \\ 
\textup{and the newly defined norm $\|\cdot\|_e$. (cf. (1.10) below)} \nonumber
\end{align}

\begin{equation}\label{eq:1.2}
\{T \}' = \{A\in B(H): AT=TA\},\ \text{the commutant of $T$.}
\end{equation}

As for the hyperinvariant subspace problem, we may and shall assume that our $T$ admidts a unit meagly generating vector $e \in H$ in the following sense satisfying (1.3) and (1.4) below.

\begin{equation}\label{eq:1.3}
\| e \| = 1,
\end{equation}

\begin{equation}\label{eq:1.4}
H = \overline{\{T \}'e},\ \text{the closure of the linear subspace $\{T \}'e$ in $H$.}
\end{equation}

Since $H$ has been assumed to be a separable infinite dimensional Hilbert space and $\{T \}'e$ is a dense subset $H$, we can find a sequence:

\begin{equation}\label{eq:1.5}
(A_i : i \in \mathbb{N}) \subset \{T \}'\ \text{satisfying}
\end{equation}

\begin{equation}\label{eq:1.6}
H = \overline{\{ A_i e : i \in \mathbb{N}\}},\ \text{the closure of the set $\{A_i e : i \in \mathbb{N} \}$ in $H$.}
\end{equation}

For every $n \in \mathbb{N}$, we put

\begin{equation}\label{eq:1.7}
E_n = \textup{the projection belonging to B(H) whose range ${\rm Range}(E_n)$ satisfying;}
\end{equation}

\begin{equation}\label{eq:1.8}
{\rm Range}\,(E_n ) = [A_i e : 1 \leq i \leq n],\ \text{the closed linear span of $\bigcup\{A_i e : 1 \leq i \leq n \}$}.
\end{equation}

By (1.6) and (1.8)

\begin{equation}\label{eq:1.9}
E_n \nearrow I \quad  (SO) \quad (n \rightarrow \infty)
\end{equation}

, where $I$ is the identity operator on $H$. 

This (1.9) enables us to conceive the following new norm, say, $\|\cdot\|_e$ on $B(H)$, given by,

\begin{equation}\label{eq:1.10}
\|A\|_e = \sum_{k \in \mathbb{N}} \frac{1}{2^{k}}  \|AE_k\|, \quad \forall A \in B(H).
\end{equation}

We also put:
\begin{equation}\label{eq:1.11}
\mathcal{A} = \textup{the abelian, SO-closed subalgebra of $B(H)$, generated  by $(E_n : n \in \mathbb{N})$.}
\end{equation}

\begin{equation}\label{eq:1.12}
\mathcal{A}_1 = \{ A \in \mathcal{A} : \| A \| \leq 1 \}
\end{equation}

\begin{align}\nonumber
\mathcal{A}_{1,0} = \Big\{ \alpha_1 (E_2 - E_1 ) + \alpha_2 (E_3 - E_2 ) + \alpha_3 (E_4 - E_3) + \cdots + \alpha_i (E_{i+1} - E_i ) + \cdots \\
: \alpha_i \in \mathbb{R}, |\alpha_i| \leq 1 , \forall i \in \mathbb{N} \Big\}. \label{eq:1.13}
\end{align}

Since every element $A$ of $\mathcal{A}_{1,0}$ is a selfadjoint operator, we have: 

\begin{equation}\label{eq:1.14}
(\textup{Kernel}\, A )^\perp = \overline{ \textup{Range}\, A}
\end{equation}

Recall that the real Banach space: $\ell^\infty = \left\{ (\alpha_i : i\in \mathbb{N}) : \alpha_i \in \mathbb{R}, {}^\forall i \in \mathbb{N} \right\}$,  
equipped with the norm: $\|(\alpha_i : i \in \mathbb{N} ) \|_\infty = \sup_{i \in \mathbb{N}} | \alpha_i | < \infty$, 
is the dual Banach space of the real Banach space 
$\ell^1 = \left\{ (\beta_i : i\in \mathbb{N}) : \beta_i \in \mathbb{R}, {}^\forall i \in \mathbb{N}\right\}$, 
equipped with the norm $\| (\beta_i: i \in \mathbb{N} ) \|_1 = \sum_{i \in \mathbb{N}} |\beta_i| < \infty$.

Clearly $(\ell^\infty)_1$ and $\mathcal{A}_{1,0}$ are in one to one correspoendence via 

\begin{equation}\label{eq:1.15}
( \alpha_i : i \in \mathbb{N}) \in (\ell^\infty)_1 \longleftrightarrow \sum_{i \in \mathbb{N}} \alpha_i (E_{i+1} - E_i ) \in \mathcal{A}_{1, 0},
\end{equation}
where  $(\ell^\infty)_1$ denotes the closed unit ball of  $\ell^\infty$. 

Hence we can transport the weak-$*$ topology $\sigma(\ell^\infty, l^1 )$ of $\ell^\infty$ induced on $(\ell^\infty)_1$ onto $\mathcal{A}_{1,0}$. 
This transported topological space will be still denoted by $\big(\mathcal{A}_{1,0} , \sigma(\ell^\infty, \ell^1)\big)$.
Thus $\big(\mathcal{A}_{1,0} , \sigma(\ell^\infty, \ell^1)\big)$ is a metrizable compact Hausdorff space (p.426 Theorem 1, p.424 Theorem 2 (Alaoglu)).

For every $n \in \mathbb{N}$, we put: $B_n = I - E_n $, i.e.

\begin{align}\label{eq:1.16}
B_n = (E_{n+1} - E_n ) + (E_{n+2} - E_{n+1} ) + (E_{n+3} - E_{n+2}) + \\
\cdots + (E_{n+k} - E_{n+k-1} )+ \cdots \quad (k \in \mathbb{N}),\ \text{and} \nonumber
\end{align}

\begin{align} \nonumber
\mathcal{A}(n) =& \Big\{ A \in \mathcal{A}_{1,0} : AE_j = 0, 1 \leq {}^\forall j \leq n, \\
&| \sum_{i \in \mathbb{N}} \| AE_i \| \beta_i | \geq | \sum_{i \in \mathbb{N}} \| B_n E_i \| \beta_i |, {}^\forall (\beta_i : i \in \mathbb{N}) \in (l^1 )_1, \nonumber \\
& \text{ such that } \beta_j = 0 ,\ 1 \leq {}^\forall j \leq n\ \text{where $(\ell^1)_1$ is the unit ball of $\ell^1$} \Big\}.\label{eq:1.17}
\end{align}
}
\end{notation}

\begin{lemma}\label{lemma:1.2}
${}^\forall n \in \mathbb{N}$,
\begin{equation}\label{eq:1.18}
B_n \in \mathcal{A}(n), 
\end{equation}
\begin{equation}\label{eq:1.19}
0 \notin \mathcal{A}(n),
\end{equation}
\begin{equation}\label{eq:1.20}
\mathcal{A}(n+1) \subset \mathcal{A}(n)\ \text{ and }
\end{equation}
\begin{equation}\label{eq:1.21}
\text{ $\big( \mathcal{A}(n), \sigma (\ell^\infty , \ell^1 ) \big)$ is a compact Hausdorff space,}
\end{equation}
\end{lemma}

\begin{proof}
The verification of \eqref{eq:1.18} is easily done. That of \eqref{eq:1.21} has already been mentioned, just after \eqref{eq:1.15}.

To verify \eqref{eq:1.19}, we observe in \eqref{eq:1.17};
\begin{eqnarray*}
\lefteqn{\sup \Big\{ | \sum_{i \in \mathbb{N}} \| A E_i \| \beta_i | : (\beta_i : i \in \mathbb{N}) \in (\ell^1)_1 \text{ such that } \beta_j = 0, 1 \leq {}^\forall j \leq n \Big \}}  \\
&\geq& \sup \Big\{ | \sum_{i \in \mathbb{N}} \| B_n E_i \| \beta_i | : (\beta_i : i \in \mathbb{N}) \in (\ell^1)_1 \text{ such that } \beta_j = 0, 1 \leq {}^\forall j \leq n \Big \} \\
&=& \sup \Big\{ | \sum_{i \in \mathbb{N}} \| B_n E_i \| \beta_i | : (\beta_i : i \in \mathbb{N}) \in (\ell^1)_1 \Big \}\  
\text{(, since $B_n E_j = 0, \hspace{0.1cm} 1 \leq {}^\forall j \leq n )$} \\
&=& \|B_n \|_\infty \text{ (, regarding $B_n$ as an element of $\ell^\infty$ as mentioned just after \eqref{eq:1.15}) } \\
&=& 1\ \text{(cf. \eqref{eq:1.16})}
\end{eqnarray*}

Hence \eqref{eq:1.19} holds.

Finally to verify \eqref{eq:1.20}, let
\begin{equation}\label{eq:1.22}
A \in \mathcal{A}(n+1), \quad \text{taken arbitrarily.}
\end{equation}

Then, by \eqref{eq:1.17}, where $n$ there, now replaced by $n+1$, we obtain:

\begin{align}\label{eq:1.23}
&A \in \mathcal{A}_{1,0} , \quad  A E_j = 0, \quad 1 \leq {}^\forall j \leq n+1 \nonumber \\
&\text{ and } \nonumber \\
&\forall (\beta_i: i \in \mathbb{N}) \in (\ell^1)_1 \text{ such that } \\
&\beta_j = 0, \quad 1 \leq {}^\forall j \leq n+1, \nonumber \\
&| \sum_{i \in \mathbb{N}} \| A E_i \| \beta_i | \geq | \sum_{i \in \mathbb{N}} \| B_{n+1} E_i \| \beta_i |. \nonumber
\end{align}

In \eqref{eq:1.17}, we notice that
\begin{equation}\label{eq:1.24}
\| B_n E_i \| = \begin{cases} 1, \quad \text{if } i \geq n+1 \\
0, \quad \text{if } 1 \leq i \leq n \end{cases}
\end{equation}

If we replace $n$ in \eqref{eq:1.24} by $n+1$, we obtain:
\begin{equation}\label{eq:1.25}
\| B_{n+1} E_i \| = \begin{cases} 1, \quad \text{if } i \geq n+2 \\
0, \quad \text{if } 1 \leq i \leq n+1 \end{cases}
\end{equation}

If we apply \eqref{eq:1.25} to \eqref{eq:1.23} we obtain:
\begin{align}\label{eq:1.26}
&A \in \mathcal{A}_{1,0} , \quad AE_j  = 0, \quad 1 \leq {}^\forall j \leq n+1 \nonumber \\
&\text{and} \nonumber \\
&\forall (\beta_i : i \in \mathbb{N}) \in (\ell^1)_1 \text{ such that } \beta_j = 0, \\
& 1 \leq {}^\forall j \leq n+1, \nonumber \\
&| \sum_{i \in \mathbb{N}} \| AE_i \| \beta_i | \geq | \sum_{n+2 \leq i < \infty} \beta_i | \nonumber
\end{align}
, equivalently,
\begin{align}\label{eq:1.27}
&A \in \mathcal{A}_{1,0} , \quad AE_j = 0, \quad 1 \leq {}^\forall j \leq n+1 \nonumber \\
&\text{and} \nonumber \\
&\forall (\beta_i : i \in \mathbb{N}) \in (\ell^1)_1 \text{ such that } \beta_j = 0, \\
& 1 \leq {}^\forall j \leq n+1, \nonumber \\
&| \sum_{i \in \mathbb{N}} \| A E_i\| \beta_i | \geq | \sum_{n \leq i < \infty} \beta_i | \nonumber \\
&\text{, since $\beta_n = \beta_{n+1} = 0$ in \eqref{eq:1.26}.}\nonumber
\end{align}

Now, $\forall(\beta_i: i \in \mathbb{N}) \in (\ell^1)_1$ such that
\begin{equation}\label{eq:1.28}
\beta_1 = \beta_2 = \cdots = \beta_n = 0,
\end{equation}
we have:
\begin{equation*}
|\sum_{i \in \mathbb{N}} \|B_n E_i \| \beta_i | = |\sum_{ n+1 \leq i < \infty} \beta_i | \quad \text{(cf. \eqref{eq:1.24}),\ i.e.}
\end{equation*}
\begin{align}\label{eq:1.29}
&|\sum_{i \in \mathbb{N}} \|B_n E_i \| \beta_i | = |\sum_{ n \leq i < \infty} \beta_i |  \\
&\text{(, since $\beta_n = 0$ in \eqref{eq:1.28}).} \nonumber
\end{align}

Then by aid of \eqref{eq:1.29}, \eqref{eq:1.27} implies the following:

\begin{align}\label{eq:1.30}
&A \in \mathcal{A}_{1,0} , \quad AE_j = 0, \quad 1 \leq {}^\forall j \leq n \nonumber \\
&\text{and} \nonumber \\
&\forall (\beta_i : i \in \mathbb{N}) \in (\ell^1)_1 \text{ such that } 1 \leq {}^\forall j \leq n \\
&\text{we have:} \nonumber \\
&| \sum_{i \in \mathbb{N}} \| A E_i  \| \beta_i | \geq | \sum_{i \in \mathbb{N}} \| B_n E_i \| \beta_i |. \nonumber
\end{align}

By \eqref{eq:1.30}, we now can say that any $A \in \mathcal{A}(n+1)$ in \eqref{eq:1.22} actually belongs to $\mathcal{A}(n)$ (cf. \eqref{eq:1.17}). This proves \eqref{eq:1.20}.
\end{proof}

\section{The solution}\label{sec:2}

\begin{theorem}\label{theorem:2.1}
Without loss of generality, let $T$ satisfy \eqref{eq:1.3} and \eqref{eq:1.4}. Then $T$ has a nontrivial closed hyperinvariant subspace whose relevant projection belongs to $\mathcal{A}$ of \eqref{eq:1.11}
\end{theorem}
\begin{proof}
By Lemma~\ref{lemma:1.2}, the family $\{ \mathcal{A}(n) : n \in \mathbb{N} \}$ consists of nonempty closed subsets of the compact Hausdorff space $( \mathcal{A}_{1,0}, \sigma(l^\infty, l^1))$, and satisfies the finite intersection property. 
Hence $\bigcap_{n \in \mathbb{N}} \mathcal{A}(n) \neq \emptyset$ ([3] p.223 Theorem 1.3(2)), so we can find an element, say, 
\begin{equation}\label{eq:2.1}
E \in \bigcap_{n \in \mathbb{N}} \mathcal{A}(n).
\end{equation}

By \eqref{eq:2.1}, \eqref{eq:1.17} and \eqref{eq:1.13}, this $E$ can be expressed as follows in many different ways depending on the various ways of choosing $n$'s $\in \mathbb{N}$, since $E \in \mathcal{A}(n)$, $\forall n \in \mathbb{N}$.

Hence we can write $E$ as follows:
\begin{align}\label{eq:2.2}
E = \alpha_{n,n} (E_{n+1}-E_n) + \alpha_{n, n+1} (E_{n+2} - E_{n+1}) + \alpha_{n, n+2} (E_{n+3} - E_{n+2}) + \cdots
\end{align}
, where $\alpha_{n,k} (\in \mathbb{R})$ satisfies $| \alpha_{n,k} | \leq 1, \hspace{0.1cm} n \leq {}^\forall k < \infty$.

Now, $\forall A \in \{T\}'$, 
\begin{align*}
&\| AE - EAE \|_e = \| (I-E) AE\|_e = \sum_{k \in \mathbb{N}} \frac{1}{2^{k}} \| (I-E) AEE_k \|\ \ (cf. \eqref{eq:1.10}) \\
&\leq 2\|A\|\sum_{k \in \mathbb{N}} \frac{1}{2^{k}} \| EE_k \| \text{(, since $\|(I-E) A\| \leq (\|I\|+\|E\|)\|A\|\leq 2\|A\|$)} \\
&\leq 2\|A\|\sum_{k \in \mathbb{N}} \frac{1}{2^{k}} \text{(, since $EE_k = 0, \forall k\leq n$ and $\|EE_k \| \leq \|E\|\|E_k \| \leq 1, \forall k \in \mathbb{N}$)} \\
&= 2\|A\|\frac{1}{2^{n}} \rightarrow 0 \quad (n \rightarrow 0)
\end{align*}
, while $E$ of \eqref{eq:2.2} is not influenced by various ways of choosing $n \in \mathbb{N}$.

Hence
\begin{equation}\label{eq:2.3}
\|AE - EAE\|_e = 0. \qquad{\rm i.e.}
\end{equation}
\begin{equation}\label{eq:2.4}
AE = EAE.
\end{equation}

By \eqref{eq:2.1} and \eqref{eq:1.13},
\begin{equation}\label{eq:2.5}
EE_1 = 0.
\end{equation}

By \eqref{eq:2.1} and \eqref{eq:1.19},
\begin{equation}\label{eq:2.6}
E \neq 0.
\end{equation}

Notice that 
by \eqref{eq:2.5} and \eqref{eq:2.6},
\begin{equation}\label{eq:2.7}
\{0 \} \subsetneq \textup{Kernel}(E) \subsetneq H.
\end{equation}

Hence,
\begin{equation}\label{eq:2.8}
\{ 0 \} \subsetneq \overline{\textup{Range}(E)} \subsetneq H,
\end{equation}
since E is a selfadjoint operator (cf. \eqref{eq:2.1}, \eqref{eq:1.17} and \eqref{eq:1.14}).

By \eqref{eq:2.8} and \eqref{eq:2.4}, $\overline{\textup{Range}(M)}$ is a desired nontrivial closed hyperinvariant subspace (of $H$) for $T$.
It is easy to see that the projection whose range is $\overline{\textup{Range}(M)}$ belongs to $\mathcal{A}$.
\end{proof}

\begin{corollary}\label{corollary:2.2}
Let $\mathcal{A}$ be an abelian von Neumann subalgebra of $B(H)$. Then $\mathcal{A}$ has a (common) nontrivial closed invariant subspace of $H$, whose relevant projection belongs to $\mathcal{A}$
\end{corollary}

\begin{remark}\label{remark:2.3}{\rm
\begin{enumerate}
\item[(i)] The operator $E$ of \eqref{eq:2.1} is uniquely determined, i.e. 
\begin{equation}\label{eq:2.9}
\bigcap_{n \in \mathbb{N}} \mathcal{A}(n) = \{ E\}.
\end{equation}

Indeed, if $F \in \bigcap_{n \in \mathbb{N}} \mathcal{A}(n)$, then $F \in \mathcal{A}(n), \forall n \in \mathbb{N}$ by \eqref{eq:1.17}.
Hence $F \in \mathcal{A}_{1,0}$ and $FE_j =0,\ 1  \leq {}^\forall j \leq n,\ \forall n \in \mathbb{N}$.

Recalling \eqref{eq:2.1} and \eqref{eq:1.17} again, $E \in \mathcal{A}_{1,0}$ and $EE_j = 0$, $1 \leq {}^\forall j \leq n$, $\forall n \in \mathbb{N}$, as well.

Thus
\begin{equation}\label{eq:2.10}
(E - F)E_j = 0, \quad  1 \leq {}^\forall j \leq n, \quad \forall n \in \mathbb{N}.
\end{equation}

By \eqref{eq:1.9}, we know that $E_j \nearrow I \quad  (SO)$ as $j \rightarrow \infty$.

Hence
\begin{equation}\label{eq:2.11}
E = F, 
\end{equation}
as desired, proving \eqref{eq:2.9}.

\item[(ii)] The analogies of Definition 2.3, Theorem 2.4, Corollary 2.5 and Corollary 2.6 in ([14] arXiv: 2002.11533v10 [math. GM] 15 Dec 2022) can be stated now in terms of meagly generating vector, hyperinvariant subspaces based on a right Hamilton space $H$.
\end{enumerate}
}
\end{remark}

\section*{Acknowledgements} 
I thank deeply Prof. Charles Akemann (UCSB), Prof. Sang Hoon Lee (CNU), Heon Lee (SNU), and my wife Soon Hee Kim.





\begin{center}
\line(1,0){350}
\end{center}
\end{document}